\documentclass[english, 11pt]{article}
\usepackage[a4paper,top=3cm, bottom=3cm, left=2.75cm, right=2.75cm]{geometry}
\usepackage{fourier}
\usepackage[T1]{fontenc}
\usepackage[utf8]{inputenc}
\usepackage{babel}
\usepackage{microtype}
\usepackage{amstext,amsmath,amssymb,amsthm}
\usepackage[round]{natbib}
\usepackage{mathrsfs}
\usepackage{authblk}
\usepackage[colorlinks=true,citecolor=blue,linkcolor=blue,urlcolor=red]{hyperref}
\usepackage[inline]{enumitem}

\theoremstyle{plain} 
\newtheorem{theorem}{Theorem}[section]

\newtheorem{lemma}[theorem]{Lemma}
\newtheorem{corollary}[theorem]{Corollary}

\theoremstyle{definition}
\newtheorem{definition}{Definition}[section]
\newtheorem{example}{Example}[section]

\let\Oldenddefinition\enddefinition
\def\enddefinition{\hfill $\triangleleft$\Oldenddefinition}%

\let\Oldendexample\endexample%
\def\endexample{\hfill $\triangleleft$\Oldendexample}%

\newcommand{\opt}{\textsl{Opt}}
\newcommand{\nb}{\textsl{NoBull}}
\newcommand{\alm}{\textsl{Almost}}
\newcommand{\be}{\mathbf{e}}
\DeclareMathOperator{\best}{best}

\begin{document}
\title{No bullying! \\ A playful proof of Brouwer's fixed-point theorem}

\author[1]{Henrik Petri\thanks{This paper is based on a chapter of Petri's PhD dissertation at Stockholm School of Economics. For helpful comments and discussions, we thank Tommy Andersson, Kim Border, Albin Erlanson, Srihari Govindan, Jean-Jacques Herings, Atsushi Kajii, Wolfgang Leininger, Bart Lipman, Piotr Ma\'{c}kowiak, Marcin Malawski, John Nachbar, Phil Reny, Al Roth, Bill Sandholm, Rajiv Sethi, Ross Starr, Lars-Gunnar Svensson, Dolf Talman, William Thomson, Rabee Tourky, J\"{o}rgen Weibull, Zaifu Yang, and two anonymous referees. Special thanks to Andy McLennan, with whom we are working on a follow-up article addressing algorithmic aspects of our approach. Financial support by the Wallander-Hedelius Foundation under grants P2014-0189:1 (Petri), P2010-0094:1, and P2016-0072:1 (Voorneveld) is gratefully acknowledged.}}

\affil[1]{Department of Economics, University of Bath, 3 East, Bath BA2 7AY, UK, \url{henrik@petri.se}}

\author[2]{Mark Voorneveld}

\affil[2]{Department of Economics, Stockholm School of Economics, Box 6501, 113 83 Stockholm, Sweden, \url{mark.voorneveld@hhs.se}}

\maketitle

\begin{abstract}
We give an elementary proof of Brouwer's fixed-point theorem. The only mathematical prerequisite is a version of the Bolzano-Weierstrass theorem: a sequence in a compact subset of $n$-dimensional Euclidean space has a convergent subsequence with a limit in that set. Our main tool is a `no-bullying' lemma for agents with preferences over indivisible goods. What does this lemma claim? Consider a finite number of children, each with a single indivisible good (a toy) and preferences over those toys. Let's say that a group of children, possibly after exchanging toys, could bully some poor kid if all group members find their own current toy better than the toy of this victim. The no-bullying lemma asserts that some group $S$ of children can redistribute their toys among themselves in such a way that all members of $S$ get their favorite toy from $S$, but they cannot bully anyone.
\end{abstract}



\bigskip \noindent Published in Journal of Mathematical Economics, 78 (2018), 1--5, 

\url{https://doi.org/10.1016/j.jmateco.2018.07.001}

\newpage
\section{Introduction}

The aim of this paper is to give a detailed, elementary proof of Brouwer's fixed-point theorem \citep[Satz 4]{Brouwer1911}: in Euclidean space $\mathbb{R}^n$, a continuous function from and to the unit simplex of nonnegative vectors with coordinates summing to one has a fixed point. The proof is accessible with a minimal mathematical background. Its main ingredient is a new `no-bullying' lemma (Lemma \ref{lemma: no bullying}) in a simple economic setting, \citeauthor{ShapleyScarf1974}'s \citeyearpar{ShapleyScarf1974} classical housing market model of agents with preferences over indivisible goods.

\cite{Park1999} gives an historical overview of many ways to prove Brouwer's fixed-point theorem; rather than repeating them here, we get straight to work.

What is the prerequisite `minimal mathematical background'? The only result our proof takes for granted is a version of the Bolzano-Weierstrass theorem: every sequence in a compact subset of $\mathbb{R}^n$, with its usual distance, has a convergent subsequence with a limit in this set.

And what is the no-bullying lemma? Consider a finite number of children, each with a single indivisible good (a toy) and preferences over those toys. Let's say that a group of children, possibly after exchanging toys, could bully some poor kid if all group members find their own current toy better than the toy of this victim. The no-bullying lemma asserts that some group $S$ of children can redistribute their toys among themselves in such a way that each member of $S$ gets his or her favorite toy from $S$ and they cannot bully anyone.

The no-bullying lemma with children caring about toys seems rather remote from Brouwer's fixed-point theorem. But the link is easier to understand after seeing how certain combinatorial proofs --- often using variants of Sperner's lemma \citep{Sperner1928} --- are structured; \citet[sec. 3]{Scarf1982} and \citet[Ch. 3--6]{Border1985} contain pedagogical accounts:

A key first step is to find solutions to a system of inequalities relating the coordinates of vectors to those of their function values. For instance, we will use that for each $\varepsilon > 0$ there is a set of vectors in $\Delta$ within distance $\varepsilon$ from each other that contains, for each coordinate $i$, an element $x$ with $x_i - \varepsilon \leq f_i(x)$. The link is provided by introducing economic agents \emph{to whom\/} it matters how large these coordinates are: we introduce toys that correspond with vectors and children that care about their coordinates --- and we're straight in the setting of the no-bullying lemma!

The second step is a standard limit argument: as $\varepsilon$ tends to zero, the Bolzano-Weierstrass theorem assures that those vectors may be chosen in such a way that they converge to a common limit $x^*$. This limit $x^*$ is the desired fixed point. Since all $x$'s tend to $x^*$ and $\varepsilon$ tends to zero, continuity of $f$ gives that $x^*_i \leq f_i(x^*)$ for all coordinates $i$. And the coordinates of $x^*$ and $f(x^*)$ both sum to one, so none of the inequalities can be strict: $x^* = f(x^*)$.

This gives us a clear road map for the remainder of our paper. We formulate the no-bullying lemma in Section \ref{sec: no-bullying implies Brouwer}. We explain the intuition why the no-bullying lemma implies Brouwer's fixed-point theorem and then prove it formally. The proof of the no-bullying lemma itself is in Section \ref{sec: proof no-bullying}. The strategy will be to slightly relax the requirements in the lemma and --- from a simple starting point --- make a series of small, explicitly defined changes until we arrive at a set of children and a reallocation of their toys that satisfy all conditions of the no-bullying lemma. Section \ref{sec: conclusion} contains concluding remarks; most of these are technical and the section can be skipped by anyone just interested in seeing how our proof works. The appendix discusses how other classical results in economic theory, the lemmas of \citet*{KKM1929} and \citet{Sperner1928}, follow from the no-bullying lemma.

\section{The no-bullying lemma and why it implies Brouwer}\label{sec: no-bullying implies Brouwer}

Consider a nonempty, finite set $I$ of children, each with a single toy. For simplicity, each child has strict preferences over toys and child $i$ starts out with toy $i$. Formally, $i$'s strict preferences are a binary relation $\succ_i$ on $I$; $x \succ_i y$ means that $i$ strictly prefers toy $x \in I$ to toy $y \in I$. These preferences are assumed to be (a) total: for all distinct $x, y \in I$, $x \succ_i y$ or $y \succ_i x$, (b) irreflexive: there is no $x$ with $x \succ_i x$, and (c) transitive: if $x \succ_i y$ and $y \succ_i z$, then also $x \succ_i z$. Since no two toys are equivalent, $i$ has a well-defined most preferred element $\best_i(Y)$ in any nonempty subset $Y$ of toys.

Say that a group of children could \emph{bully\/} another child if its members agree that this victim has a lousy toy: each group member finds his or her toy better than the victim's toy. According to the no-bullying lemma, some group $Y$ of children can exchange their toys among themselves in such a way that all members of $Y$ receive their best toy from $Y$, yet they cannot bully anyone:

\begin{lemma}[No-bullying lemma]\label{lemma: no bullying}
Let $I$ be a nonempty, finite set. For each $i \in I$, let $\succ_i$ be a strict preference relation on $I$. Then there is a nonempty $Y \subseteq I$ satisfying:
\begin{enumerate}
\item[(a)] Optimality: each $i \in Y$ can get her most preferred element of $Y$, i.e., $\{\best_i(Y): i \in Y\} = Y$.
\item[(b)] No Bullying: there is no $x \in I$ with\/ $\best_i(Y) \succ_i x$ for all $i \in Y$.
\end{enumerate}
\end{lemma}

Observe that No Bullying implies Optimality: if there were an $x \in Y \setminus \{\best_i(Y): i \in Y\}$, then $\best_i(Y) \succ_i x$ for all $i \in Y$, contradicting No Bullying. But it is convenient to state both properties explicitly: we often refer to Optimality in our proofs.

We prove the no-bullying lemma in the next section. It easily extends to distinct sets $I$ of children and $T$ of toys, given some endowment $\ell: T \to I$ mapping toys to their initial owner. The trick is to identify both the set of children and the set of toys with $\{(i, t) \in I \times T: i = \ell(t)\}$. So toys are labeled with their initial owners. And we replace child $i$ with replicas, one for each toy she owns, with the same preferences as $i$.

\begin{corollary}\label{cor: no bullying}
Let each $i$ in a nonempty, finite set $I$ of children have strict preferences $\succ_i$ over a nonempty, finite set $T$ of toys. A function $\ell: T \to I$ maps each toy $t \in T$ to its initial owner $\ell(t) \in I$. Then there are nonempty subsets $C$ of children and $E$ of toys they can exchange satisfying:
\begin{itemize}
\item[(a)] Ownership: the members of $C$ own the elements of $E$, i.e., $C = \{\ell(t): t \in E\}$.
\item[(b)] Optimality: each $i \in C$ can get her most preferred element of $E$, i.e., $\{\best_i(E): i \in C\} = E$.
\item[(c)] No Bullying: there is no $t \in T$ with\ $\best_i(E) \succ_i t$ for all $i \in C$.
\end{itemize}
\end{corollary}
\begin{proof}
Let $I^* = \{(\ell(t), t): t \in T\}$. For each $(\ell(t), t) \in I^*$, define preferences $\succ_{(\ell(t), t)}$ on $I^*$ such that
\begin{equation}\label{eq: replicated preferences}
(\ell(t), t) \text{ inherits $\ell(t)$'s preferences:} \qquad (\ell(t_1), t_1) \succ_{(\ell(t), t)} (\ell(t_2), t_2) \iff t_1 \succ_{\ell(t)} t_2.
\end{equation}
By Lemma \ref{lemma: no bullying}, a nonempty $Y \subseteq I^*$ satisfies Optimality and No Bullying. Let $E = \{t: (\ell(t), t) \in Y\}$ be the exchanged toys and $C = \{ \ell(t): (\ell(t), t) \in Y\} = \{\ell(t): t \in E\}$ their owners. By \eqref{eq: replicated preferences}, the best toy of $(\ell(t), t) \in Y$ from $Y$ is the best toy of $\ell(t) \in C$ from $E$. So the Optimality and No Bullying properties for $C$ and $E$ simply rephrase the corresponding properties of $Y$.
\end{proof}

We use the no-bullying lemma to find approximations of fixed points:

\begin{lemma}\label{lemma: AFP}
Let $f$ be a continuous function from and to $\Delta = \{x \in \mathbb{R}^n: x_1, \ldots, x_n \geq 0 \text{ and } \sum_{i=1}^n x_i = 1\}$ in some $n$-dimensional Euclidean space \, $\mathbb{R}^n$. For each $\varepsilon > 0$ there is a set $T(\varepsilon) \subseteq \Delta$ such that
\begin{align}
&\text{for all $x$ and $y$ in $T(\varepsilon)$: \quad $\max_i \, \lvert x_i - y_i \rvert < \varepsilon$,} \label{eq: nearby} \\
&\text{for all $i \in \{1, \ldots, n\}$ there is an $x$ in $T(\varepsilon)$ with $x_i - \varepsilon \leq f_i(x)$.} \label{eq: AFP}
\end{align}
\end{lemma}

Before we give the intuition behind Lemma \ref{lemma: AFP} and prove it, we show that $f$ has a fixed point via a standard limit argument. Let $\varepsilon_1, \varepsilon_2, \ldots$ be a sequence of positive numbers converging to zero. For each such $\varepsilon_m > 0$, take a set $T(\varepsilon_m)$ as in the lemma and an element $x(m) \in T(\varepsilon_m)$. These $x(1), x(2), \ldots$ lie in the compact set $\Delta$. By the Bolzano-Weierstrass theorem we may pass to a subsequence if necessary and assume that they converge to a limit $x^* \in \Delta$. By \eqref{eq: nearby}, the distance between the elements of $T(\varepsilon_m)$ tends to zero as $\varepsilon_m \to 0$. So \emph{all\/} sequences obtained by assigning to each integer $m$ an element of $T(\varepsilon_m)$ converge to $x^*$. Taking limits, \eqref{eq: AFP} and continuity of $f$ then give that $x^*_i \leq f_i(x^*)$ for each $i$. Since the coordinates of $x^*$ and $f(x^*)$ both sum to one, these weak inequalities must be equalities: $x^* = f(x^*)$, proving Brouwer's fixed-point theorem.

But it looks like a huge step from the no-bullying lemma to Lemma \ref{lemma: AFP}. What is the intuition? Fix a large, finite subset of $\Delta$ by giving coordinates only finitely many values. This grid is our set of toys. Introduce one child for each coordinate $i \in \{1, \ldots, n\}$. Child $i$ prefers points with small $i$-th coordinates. Now apply Corollary \ref{cor: no bullying}. With suitable initial endowments, the set of exchanged toys satisfies the conditions on approximate fixed points in Lemma \ref{lemma: AFP}. Roughly speaking, if the no-bullying lemma gives each child $i$ in subset $C$ a toy, Optimality and $i$ liking small $i$-th coordinates imply that its $i$-th coordinate can't be very large. But such coordinates can't be very small either: then it were possible to define a vector where all coordinates $i \in C$ are a bit larger. So all members of $C$ find this vector strictly worse, contradicting No Bullying. This doesn't leave much room to manoeuver in: the exchanged toys lie near each other, as \eqref{eq: nearby} says.

The proof sketch hasn't mentioned function $f$ yet, let alone why \eqref{eq: AFP} holds. The initial endowment provides this link. Each toy changing hands is owned by some child in $C$. The endowment is chosen such that if $i \in C$ owns $x$, then $x_i \leq f_i(x)$. This clearly implies the inequality $x_i - \varepsilon \leq f_i(x)$ in \eqref{eq: AFP}. And if $i$ lies outside $C$, the left side of \eqref{eq: AFP} turns out to be nonpositive for all exchanged toys, but its right side is nonnegative: function values lie in $\Delta$. Now the formal proof:

\begin{proof}[Proof of Lemma \ref{lemma: AFP}]
Let $\varepsilon > 0$. Let positive integer $N$ satisfy $2n/N < \varepsilon$. Apply Corollary \ref{cor: no bullying} to toys
\[
T = \{x \in \Delta: x_i \in \{0/N, 1/N, \ldots, N/N\} \text{ for all } i = 1, \ldots, n\},
\]
elements of $\Delta$ whose coordinates are multiples of $1/N$, and children $I = \{1, \ldots, n\}$, one per coordinate. For each $i \in I$, let $\succ_i$ be any strict preference on $T$ where $i$ prefers smaller $i$-th coordinates: for all $x, y \in T$, if $x_i < y_i$, then $x \succ_i y$. Such preferences are not unique: $i$ may order vectors with identical $i$-th coordinates arbitrarily.\footnote{E.g., $i$ may have lexicographic preferences and look at the coordinates in some fixed order, starting with coordinate $i$. Let $x \succ_i y$ if $x_j \neq y_j$ for some coordinate $j$ and, in the fixed order of the coordinates, the first such $j$ has $x_j < y_j$.}
If $x \in T$, there is an $i$ with $x_i \leq f_i(x)$, since the coordinates of $x$ and $f(x)$ both sum to 1. Define endowments $\ell: T \to I$ for each $x \in T$ by $\ell(x) = \min \{i \in \{1, \ldots, n\}: x_i \leq f_i(x)\}$.

By Corollary \ref{cor: no bullying}, there are subsets $C$ of $I$ and $E$ of $T$ satisfying Ownership, Optimality, and No Bullying. We show that Lemma \ref{lemma: AFP} holds if we take $T(\varepsilon) := E$, the set of exchanged toys.

For $i \in C$, let $\beta(i) := \best_i(E)$ be $i$'s favorite toy in $E$. We first prove 3 observations:
\begin{enumerate*}[label=(O\arabic*)]
\item\label{O1} for each $x \in E$ and $i \in C$, $\beta_i(i) \leq x_i$;
\item\label{O2} $\sum_{i \in C} \beta_i(i) > 1 - \frac{n}{N}$, and
\item\label{O3} for each $x \in E$ and $i \notin C$, $0 \leq x_i < \frac{n}{N}$.
\end{enumerate*}

\ref{O1} holds because $i \in C$ finds $\beta(i)$ better than all other elements of $E$ and $i$ likes small $i$-th coordinates. For \ref{O2}: if, to the contrary, $\sum_{i \in C} \beta_i(i) \leq 1 - \frac{n}{N}$, then $0 \leq \sum_{i \in C} \left(\beta_i(i) + \frac{1}{N} \right) \leq 1$. So there is an $x \in T$ with $x_i = \beta_i(i) + \frac{1}{N}$ if $i \in C$. Then $\beta_i(i) < x_i$, so $\beta(i) \succ_i x$ for all $i \in C$, contradicting No Bullying. \ref{O3} follows from \ref{O1} and \ref{O2}: if $x \in E$ and $i \notin C$, then
\[
0 \leq x_i \leq \sum_{m \notin C} x_m = 1 - \sum_{m \in C} x_m \stackrel{\scriptstyle \ref{O1}}{\leq} 1 - \sum_{m  \in C} \beta_m(m) \stackrel{\scriptstyle \ref{O2}}{<} \frac{n}{N}.
\]
Now \eqref{eq: nearby} holds: Let $x, y \in E$ and $i \in I$. We show that $\lvert x_i - y_i \rvert < \varepsilon$. If $i \in C$, then
\[
0 \stackrel{\scriptstyle \ref{O1}}{\leq} x_i - \beta_i(i) \leq \sum_{m \in C} \left( x_m - \beta_m(m) \right) \leq 1 - \sum_{m \in C} \beta_m(m) \stackrel{\scriptstyle \ref{O2}}{<} \frac{n}{N} < \frac{\varepsilon}{2},
\]
and similarly for $y$. By the triangle inequality, $\lvert x_i - y_i \rvert \leq \lvert x_i - \beta_i(i) \rvert + \lvert \beta_i(i) - y_i \rvert < \varepsilon$. If $i \notin C$, then the triangle inequality and \ref{O3} give $\lvert x_i - y_i \rvert \leq \lvert x_i \rvert + \lvert y_i \rvert < 2n/N < \varepsilon$.

Also \eqref{eq: AFP} holds: Let $i \in \{1, \ldots, n\}$. If $i \in C$, then by Ownership, there is an $x \in E$ with $\ell(x) = i$, so $x_i - \varepsilon < x_i \leq f_i(x)$. If $i \notin C$, then for each $x \in E$, $x_i \leq n/N < \varepsilon$ by \ref{O3}. So $x_i - \varepsilon < 0 \leq f_i(x)$.
\end{proof}

\medskip
\noindent \textbf{Some remarks about the top-trading cycle algorithm:} We formulated the no-bullying lemma in familiar economic vocabulary: agents with preferences over indivisible goods, as in \citet{ShapleyScarf1974}. Their article is famous for introducing the top-trading-cycle (TTC) algorithm for socially stable allocations of indivisible goods, so the question naturally arises whether the resulting allocations have anything in common.

On the positive side, Optimality says that owners \emph{within a subset $S$\/} redistribute their toys so that each member receives her most preferred one: if you restrict attention to the subproblem reduced to agents and their toys in $S$, the whole market clears already in the very first iteration of the TTC algorithm with everybody obtaining their favorite item. But you cannot simply ignore the agents outside $S$ and this makes the connection with the TTC algorithm only superficial: in the following example, no member of the only set $S$ satisfying the conditions of the no-bullying lemma gets the same item as in the TTC algorithm.

\begin{example}
Consider three children with preferences over their toys as follows:
\[
2 \succ_1 1 \succ_1 3, \qquad 3 \succ_2 2 \succ_2 1, \qquad 2 \succ_3 1 \succ_3 3.
\]
In the first iteration of the TTC algorithm, 2 and 3 exchange toys. In the second iteration, 1 is stuck with toy 1. But neither $S = \{2, 3\}$ nor $S = \{1\}$ satisfies No Bullying: 2 and 3 can bully 1, since
\[
\best_2(\{2,3\}) = 3 \succ_2 1 \qquad \text{and} \qquad \best_3(\{2,3\}) = 2 \succ_3 1.
\]
Likewise, 1 can bully 3, since $\best_1(\{1\}) = 1 \succ_1 3$.  Checking the remaining candidates, it follows that only $S = \{3\}$ --- child 3 holding on to toy 3 --- satisfies the conditions of the no-bullying lemma. But in the TTC algorithm, child 3 was allocated toy 2 instead of toy 3.
\end{example}

Intuitively, if you get an item in an iteration of the TTC algorithm, then it is the remaining item you like most. Whether you are a potential target for bullying is about something different, namely whether \emph{others\/} find your item worse than their allotment. They have little to do with each other. The no-bullying lemma does not aim for efficiency: Optimality requires owners \emph{within a subset\/} $S$ to redistribute their toys optimally, but No Bullying pushes in the opposite direction: members of $S$ shouldn't get toys that are too good, because then it is easier to find a target for bullying. So the lemma is not a normative principle of distributive justice. It is just a useful tool in our proofs, cast in a language that hopefully makes it easier to remember.

\section{Proof of the no-bullying lemma}\label{sec: proof no-bullying}

Fix a nonempty finite set $I$ of children/toys and strict preferences $\succ_i$ over $I$ for each child $i \in I$. A pair $(Y, Z)$ of nonempty subsets $Y$ (children) and $Z$ (toys) of $I$ that satisfies
\begin{alignat}{2}
& \text{Optimality (\opt):}  && \{\best_i(Z): i \in Y\} = Z, \label{eq: Opt} \\
& \text{No Bullying (\nb):}\quad && \text{there is no $x \in I$ with $\best_i(Z) \succ_i x$ for all $i \in Y$,} \label{eq: NoBu}
\intertext{and $Y = Z$ proves the no-bullying lemma \ref{lemma: no bullying}: just plug $Y = Z$ into \eqref{eq: Opt} and \eqref{eq: NoBu}. To ease the search, allow pairs where $Y$ and $Z$ are almost the same --- the number of elements $|Y \setminus Z|$ of $Y \setminus Z$ is small:}
& \text{Almost (\alm):} \quad     && |Y \setminus Z| \leq 1. \notag
\end{alignat}
Pair $(Y, Z)$ is a \emph{candidate\/} if it satisfies \alm\ and \nb. Just as in Lemma \ref{lemma: no bullying}, each candidate satisfies \opt: if there were an $x \in Z \setminus \{\best_i(Z): i \in Y\}$, then $\best_i(Z) \succ_i x$ for all $i \in Y$, contradicting \eqref{eq: NoBu}. We find a candidate with $Y = Z$ by starting with a simple candidate and making a series of one-element adjustments.

A neighbor of candidate $(Y,Z)$ is a candidate obtained by adding or removing a single element in exactly one of the sets $Y$ or $Z$. Formally, the symmetric difference of sets $A$ and $B$ is $A \bigtriangleup B = (A \setminus B) \cup (B \setminus A)$. Candidates $(Y,Z)$ and $(Y',Z')$ are \emph{left neighbors\/} if $|Y \bigtriangleup Y'| = 1$ and $Z = Z'$; they are \emph{right neighbors\/} if $Y = Y'$ and $|Z \bigtriangleup Z'| = 1$; they are \emph{neighbors\/} if they are left or right neighbors.

We often use \emph{monotonicity\/}: if $(Y,Z)$ satisfies \alm, then so do $(Y', Z)$ for any smaller $Y' \subseteq Y$ and $(Y,Z')$ for any larger $Z' \supseteq Z$. And if $(Y,Z)$ satisfies \nb, then so do $(Y',Z)$ for any larger $Y' \supseteq Y$ and $(Y,Z')$ for any smaller $\emptyset \neq Z' \subseteq Z$. Only the last point is nontrivial, but if members of $Y$ can't bully using the best elements in the large set $Z$, they surely can't with those in subset $Z'$.

Lemmas \ref{lemma: singleton contender} to \ref{lemma: Y larger than Z} characterize the neighbors of candidates $(Y,Z)$ with $Y \neq Z$. \opt\ gives $|Y| \geq |\{\best_i(Z): i \in Y\}| = |Z|$. So by \alm, $Y$ has $|Z|$ or $|Z|+1$ elements. Lemma \ref{lemma: singleton contender} treats the candidates with $|Y| = |Z| = 1$ and Lemma \ref{lemma: Y and Z equal size} those with $|Y| = |Z| \geq 2$; in the latter case, $Y \neq Z$ and \alm\ imply that $Y \setminus Z$ and $Z \setminus Y$ are singletons. In particular, $Z \setminus Y = \{k\}$ for some $k \in I$. Finally, Lemma \ref{lemma: Y larger than Z} addresses candidates with $|Y| = |Z|+1$; by \alm, $Z$ is then a proper subset of $Y$. And since $Y$ has one element more than $\{\best_i(Z): i \in Y\} = Z$, there is a unique pair $j_1$ and $j_2$ of distinct elements in $Y$ with the same best element of $Z$: $\best_{j_1}(Z) = \best_{j_2}(Z)$.

For $i \in I$, let $w_i$ be $i$'s worst element of $I$: $x \succ_i w_i$ for all $x \neq w_i$. Let $B_i = (\{i\}, \{w_i\})$.

\begin{lemma}\label{lemma: singleton contender}
For each $i \in I$, $B_i$ is the unique candidate whose first component is $\{i\}$.  If $i \neq w_i$, then its unique neighbor is $(\{i,w_i\},\{w_i\})$.
\end{lemma}

\begin{proof}
$B_i$ is a candidate: clearly, \alm\ holds; \nb\ holds by definition of $w_i$. If also $(\{i\},Z)$ is a candidate, $Z$ is a singleton by \opt. And by \nb, this singleton must be $\{w_i\}$.

Suppose that $i \neq w_i$. $B_i$ has no right neighbor as it is the only candidate with first component $\{i\}$. If $(Y, \{w_i\})$ is a left neighbor, then $Y$ cannot be empty. To satisfy \alm, the only element we can add to $\{i\}$ is $w_i$. By monotonicity, this $(\{i,w_i\}, \{w_i\})$ also satisfies \nb.
\end{proof}

\begin{lemma}\label{lemma: Y and Z equal size}
If candidate $(Y,Z)$ has $|Y| = |Z| \geq 2$ and $Z \setminus Y = \{k\}$ for some $k \in I$, then $(Y, Z \setminus \{k\})$ and $(Y \cup \{k\}, Z)$ are its only neighbors.
\end{lemma}

\begin{proof}
If $(Y, Z')$ is a right neighbor, then $Z' = Z \setminus \{k\}$: since $|Y| = |Z|$, adding an element to $Z$ would contradict \opt. And to satisfy \alm, the only element we can remove is $k$. By monotonicity, $(Y, Z \setminus \{k\})$ satisfies \nb.

Likewise, if $(Y',Z)$ is a left neighbor, then $Y' = Y \cup \{k\}$: since $|Y| = |Z|$, removing an element from $Y$ would contradict \opt. And to satisfy \alm, the only element we can add is $k$. By monotonicity, $(Y \cup \{k\}, Z)$ satisfies \nb.
\end{proof}

\begin{lemma}\label{lemma: Y larger than Z}
If candidate $(Y,Z)$ has $|Y| = |Z| + 1$, let $j_1$ and $j_2$ be the unique pair of distinct elements of $Y$ with $\best_{j_1}(Z) = \best_{j_2}(Z)$. For $j \in \{j_1, j_2\}$, let the possible bullying victims of $Y \setminus \{j\}$ be
\[
V_j := \{x \in I: \best_i(Z) \succ_i x \text{ for all } i \in Y \setminus \{j\}\}.
\]
If $V_j = \emptyset$, then $(Y \setminus \{j\}, Z)$ is a neighbor. If $V_j \neq \emptyset$ and $x_j$ is its $\succ_j$-worst element, then $(Y, Z \cup \{x_j\})$ is a neighbor. This describes two neighbors, one for each $j \in \{j_1, j_2\}$. There are no others.
\end{lemma}

\begin{proof}
If $(Y', Z)$ is a left neighbor, then $Y' = Y \setminus \{j\}$ for some $j \in \{j_1, j_2\}$: removing any other element of $Y$ would violate \opt\ and adding an element to $Y$ would violate \alm. This $(Y \setminus \{j\}, Z)$ satisfies \alm\ by monotonicity. It satisfies \nb\ if and only if $V_j = \emptyset$.

We prove: $(Y, Z')$ is a right neighbor if and only if $Z' = Z \cup \{x_j\}$ for some $j \in \{j_1, j_2\}$ with $V_j \neq \emptyset$.

A right neighbor $(Y, Z')$ adds an element, say $t$, to $Z$: removing one violates \alm. Now $|Y| = |Z'|$, so by \opt\ each element of $Z'$ is most preferred in $Z'$ by a single $i \in Y$. As $j_1$ and $j_2$ no longer have the same favorite, one $j \in \{j_1, j_2\}$ has the added $t$ as the most preferred element and all other $i \in Y \setminus \{j\}$ still prefer an element of $Z$ to $t$. So $t \in V_j$. By \nb\ of $(Y, Z \cup \{t\})$, there is no $x$ with $\best_i(Z \cup \{t\}) = \best_i(Z) \succ_i x$ for all $i \in Y \setminus \{j\}$ and $\best_j(Z \cup \{t\}) = t \succ_j x$, i.e., no $x \in V_j$ with $t \succ_j x$. So $t \in V_j$ is $j$'s worst element $x_j$ of $V_j$.

It remains to show that $(Y, Z \cup \{x_j\})$ is a candidate for all $j \in \{j_1, j_2\}$ with $V_j \neq \emptyset$. \alm\ holds by monotonicity. Does \nb\ hold? Since $x_j \in V_j$, each $i \in Y \setminus \{j\}$ has $\best_i(Z \cup \{x_j\}) = \best_i(Z)$. Then $\best_j(Z \cup \{x_j\}) = x_j$: otherwise $\best_i(Z) \succ_i x_j$ for all $i \in Y$, contradicting \nb\ of $(Y,Z)$. If $(Y, Z \cup \{x_j\})$ violates \nb, there is an $x \in I$ with $\best_i(Z \cup \{x_j\}) = \best_i(Z) \succ_i x$ for all $i \in Y \setminus \{j\}$ and $\best_j(Z \cup \{x_j\}) = x_j \succ_j x$. So $x \in V_j$ and $x_j \succ_j x$. But $x_j$ is $j$'s worst member of $V_j$, a contradiction.
\end{proof}

If $V_j \neq \emptyset$ for both $j \in \{j_1, j_2\}$, the members of $Y \setminus \{j\}$ can bully $x_j$. But then $x_{j_1} \neq x_{j_2}$: otherwise $(Y,Z)$ violates \nb. In particular, Lemma \ref{lemma: Y larger than Z} always produces exactly two neighbors.

We finally show that there is a candidate $(Y,Z)$ with $Y = Z$. Fix $i \in I$. Look at candidate $(\{i\}, \{w_i\})$ in Lemma \ref{lemma: singleton contender}. If $i = w_i$, we're done. If not, move to its only neighbor $(\{i, w_i\}, \{w_i\})$. Now proceed recursively: as long as we're at a candidate $(Y,Z)$ with $|Y| \geq 2$ and $Y \neq Z$, Lemma \ref{lemma: Y and Z equal size} or \ref{lemma: Y larger than Z} assures that $(Y,Z)$ has exactly two neighbors. So we can move to a neighbor other than the one we just came from. Suppose this path of neighboring candidates does \emph{not\/} reach a candidate $(Y,Z)$ with $Y = Z$. As the set of candidates is finite, only two things can happen: (1) the path cycles, revisiting a candidate we encountered before; or (2) we come to a candidate without two neighbors, i.e., a candidate $(Y,Z)$ where $Y$ is a singleton. But both lead to a contradiction:

If it cycles, pick the first candidate $(Y,Z)$ to be revisited. Can it have $|Y| \geq 2$? It can only be reached via one of its two neighbors. But that neighbor was passed on the first visit as well, either going to $(Y,Z)$ or coming from $(Y,Z)$. So $(Y,Z)$ is not the first to be revisited. Likewise, it can't have $|Y| = 1$, which can only be arrived at via its unique neighbor. So cycling gives a contradiction.

If we reach a candidate with singleton first component, it is of the form $(\{j\}, \{w_j\})$ for some $j \in I$ by Lemma \ref{lemma: singleton contender}. We assumed $j \neq w_j$. Reviewing the statements of Lemma \ref{lemma: singleton contender} to \ref{lemma: Y larger than Z}, note that if candidate $(Y,Z)$ with $Y \neq Z$ has neighbor $(Y',Z')$, then $Y' \setminus Z' \subseteq Y \setminus Z$. So each candidate $(Y',Z')$ on the path from $(\{i\}, \{w_i\})$ has $Y' \setminus Z' \subseteq \{i\}$. Since $\{j\} \setminus \{w_j\} = \{j\}$, it follows that $j = i$: the path returns to $(\{i\}, \{w_i\})$. A contradiction, since we ruled out cycles.

\section{Concluding remarks}\label{sec: conclusion}

\textbf{General:} The goal of this paper was to give a very elementary proof of Brouwer's fixed-point theorem: if you know Bolzano-Weierstrass, you're good to go! The proof follows the strategy of standard combinatorial proofs (see \citeauthor{Scarf1982}'s \citeyearpar{Scarf1982} overview) to find vectors satisfying a system of inequalities and then apply a limit argument. But in contrast with such proofs, our no-bullying lemma that helps to produce the inequalities requires no knowledge about simplicial subdivisions or triangulations, the facial structure of polytopes, and the dimension and boundaries of such faces relative to suitably chosen affine hulls. With the no-bullying lemma, it is possible to prove the fixed-point theorem rigorously at an early stage of the undergraduate curriculum.

We purposely framed the lemma as a playful story about children and toys that hopefully makes it easier to remember.

\textbf{Proof variants:} It might seem inefficient to provide two versions of the no-bullying lemma. But proving the simpler version (Lemma \ref{lemma: no bullying}) and deriving the second as a corollary (Corollary \ref{cor: no bullying}) requires substantially less cumbersome notation. Similarly, appropriately rephrased versions of the no-bullying lemma hold if we allow weak instead of strict preferences, but our simpler case already produces the results we need.

To prove existence of approximate fixed points in Lemma \ref{lemma: AFP}, we applied the no-bullying lemma to a sufficiently fine grid of points in the simplex. For simplicity, we gave each coordinate finitely many feasible values $0$, $1/N$, \ldots, $N/N$, for some large integer $N$. But there is considerable freedom. A sufficient condition for a grid $T$ to contain such an approximate fixed point is that for each $z \in \mathbb{R}^n_+$ with $\sum_j z_j \leq 1 - \varepsilon/2$ there is a $t \in T$ with $z_j < t_j$ for all $j$. Indeed, from property \ref{O2} in Lemma \ref{lemma: AFP}, we want a particular inequality: $\sum_{i \in C} \beta_i(i) > 1 - \varepsilon/2$. The sufficient condition on $T$ assures this, but looks less appetizing than providing a natural grid explicitly.

\textbf{Topics of ongoing work:} A useful consequence of the freedom to choose a grid is that if one grid works, then so does any finer grid. To fine-tune algorithms searching for better approximations of a fixed point, this option to add grid points wherever we please can be an advantage over traditional simplicial algorithms where points in the grid typically need to satisfy additional topological/affine independence/nondegeneracy assumptions. So even though the points in our grid happen to be the vertices of a simplicial subdivision \citep[p.~1240]{Kuhn1968}, this is of no relevance to our proof.

We leave it to subsequent work to address such algorithmic aspects in detail. We do briefly mention two results. Firstly, we proved the no-bullying lemma using a path-following algorithm. This algorithm can be shown to belong to a complexity class PPAD (Polynomial Parity Argument for Directed graphs) introduced in \citet{Papadimitriou1994}. Problems in that complexity class are typically geometric, searching for a point in a Euclidean space satisfying certain conditions. In that sense, the no-bullying problem is a distinctive member of the PPAD class, coming from a purely combinatorial setting of agents with preferences over finitely many alternatives.

Secondly, Scarf's algorithm finds approximate fixed points in so-called primitive sets; \citet{Tuy1979} generalizes this approach in a more abstract setting and finds `completely labeled primitive sets' with properties similar to those satisfying the conditions of the no-bullying lemma with distinct sets of children and toys.\footnote{More precisely: in the notation of Corollary \ref{cor: no bullying}, he also considers a finite collection $(\succ_i)_{i \in I}$ of preferences over a set $T$, but needs to extend those preferences to $T \cup I$. If one reverts the preferences on $T$ --- he characterizes his sets in terms of least preferred elements, whereas we are interested in most preferred ones --- then a pair $(C, E)$ satisfying the conditions of our Corollary \ref{cor: no bullying} corresponds with a completely labeled primitive set $U := E \cup (I \setminus C) \subset T \cup I$ in \citet{Tuy1979}.} His path-following algorithm and ours are similar in the sense of obtaining the desired sets using a series of small changes, but distinct in other senses.\footnote{For instance, primitive sets always have the same number of elements, whereas the size $|Y| + |Z|$ of candidates $(Y, Z)$ in our proof may change. Moreover, primitive sets do not refer to ownership/labels, whereas candidates --- through \alm\ --- do. The latter helps to keep the graph relatively small.}

\appendix
\section{Appendix: The KKM Lemma}

Also another classical result in economic theory, the eponymous KKM lemma of \citet{KKM1929}, follows directly from the no-bullying lemma. Their proof used Sperner's lemma; \citet[Sec. 9]{Border1985} gives two proofs using Brouwer. As before, fix $n \in \mathbb{N}$ and $\Delta = \{x \in \mathbb{R}^n: x_1, \ldots, x_n \geq 0 \text{ and } \sum_{i=1}^n x_i = 1\}$. For each nonempty $J \subseteq \{1, \ldots, n\}$, let $\Delta_J = \{x \in \Delta: x_j = 0 \text{ for all } j \in \{1, \ldots, n\} \setminus J\}$.

\begin{theorem} [KKM Lemma]
If $X_1, \ldots, X_n$ are closed subsets of $\Delta$ such that $\Delta_J \subseteq \bigcup_{j \in J} X_j$ for each nonempty $J \subseteq \{1, \ldots, n\}$, then $\bigcap_{i=1}^n X_i \neq \emptyset$.
\end{theorem}

\begin{proof}
It suffices to show that for each $\varepsilon > 0$ there is a set $S(\varepsilon) \subseteq \Delta$ with
\begin{enumerate*}[label=\emph{(\roman*)}]
\item\label{kkm1} an element in each $X_i$ and
\item\label{kkm2} for all $x, y \in S(\varepsilon)$, $\max_i \lvert x_i - y_i \rvert < \varepsilon$.
\end{enumerate*}
Indeed, let positive $\varepsilon_1, \varepsilon_2, \ldots$ tend to zero and let $S(\varepsilon_1), S(\varepsilon_2), \ldots$ be such corresponding sets. By \ref{kkm1}, for each $\varepsilon_m > 0$ and coordinate $i$, there is an $x_{i,m} \in S(\varepsilon_m) \cap X_i$. Sequence $(x_{1,m})_{m}$ lies in compact set $\Delta$. By Bolzano-Weierstrass, we may pass to a subsequence if necessary and assume it converges to some $x^* \in \Delta$. It also lies in $X_1$, which is closed. So $x^* \in X_1$. By \ref{kkm2}, the distance between the elements of $S(\varepsilon_m)$ goes to zero as $\varepsilon_m \to 0$. So for \emph{each\/} coordinate $i$: $x_{i,m} \to x^*$ and $x^* \in X_i$. Hence, $x^* \in \bigcap_{i=1}^n X_i$.

So let $\varepsilon > 0$. Let positive integer $N$ have $2n/N < \varepsilon$. Apply Corollary \ref{cor: no bullying} to $I = \{1, \ldots, n\}$ and
\[
T = \{x \in \Delta: x_i \in \{0/N, 1/N, \ldots, N/N\} \text{ for all } i = 1, \ldots, n\}.
\]
For each $i \in I$, let strict preferences $\succ_i$ be such that $x \succ_i y$ whenever $x_i > y_i$: $i$ prefers larger $i$-th coordinates. For $x \in T$, let $J_x = \{j \in I: x_j > 0\}$. Then $x \in \Delta_{J_x} \subseteq \cup_{j \in J_x} X_j$, so there is a $j$ with $x_j > 0$ and $x \in X_j$. Define $\ell: T \to I$ by letting $\ell(x)$ be any such $j$. By Corollary \ref{cor: no bullying}, a pair of subsets $C$ of $I$ and $E$ of $T$ satisfy Ownership, Optimality, and No Bullying. We show that $S(\varepsilon) := E$ satisfies \ref{kkm1} and \ref{kkm2}.

\ref{kkm1}: Each $i \in C$ gets her best toy $\beta(i) := \best_i(E)$ in $E$: $\beta_i(i) \geq x_i$ for all $x \in E$. If $\beta_i(i) = 0$ for some $i \in C$, then $x_i = 0$ and hence $\ell(x) \neq i$ for each $x \in E$, contradicting Ownership. So $\beta_i(i) > 0$ for all $i \in C$. If $k \in I \setminus C$, each $i \in C$ then prefers $\beta(i)$ to the $k$-th standard basis vector $\be_k$ of $\mathbb{R}^n$, contradicting No Bullying. So $I$ equals $C = \{\ell(x): x \in E\}$: for each $i \in I$, some $x \in E$ has $\ell(x) = i$. Hence $x \in X_i$.

\ref{kkm2}: For each $i \in I$, $\beta_i(i) > 0$ and $\beta(i) \in T$ imply $\beta_i(i) \geq \frac{1}{N}$. Also, $\sum_{i \in I} \beta_i(i) < 1 + \frac{n}{N}$: otherwise, $\sum_{i \in I} \left(\beta_i(i) - \frac{1}{N}\right) \geq 1$, so there is an $x \in T$ with $0 \leq x_i \leq \beta_i(i) - \frac{1}{N}$ for all $i$. Then $x_i < \beta_i(i)$, so $\beta(i) \succ_i x$ for all $i \in I$, contradicting No Bullying. For all $i, j \in I$, Optimality implies $\beta_i(i) \geq \beta_i(j)$, so
\[
0 \leq \beta_i(i) - \beta_i(j) \leq \sum_{k \in I} \left(\beta_k(k) - \beta_k(j)\right) = \sum_{k \in I} \beta_k(k) - 1 < \frac{n}{N}.
\]
Finally, let $x, y \in E$ and $i \in I$. By Optimality, there are $j, k \in I$ with $x = \beta(j)$ and $y = \beta(k)$, so
\[
\lvert x_i - y_i \rvert = \lvert \beta_i(j) - \beta_i(k) \rvert \leq \lvert \beta_i(j) - \beta_i(i) \rvert + \lvert \beta_i(i) - \beta_i(k) \rvert < \frac{2n}{N} < \varepsilon. \qedhere
\]
\end{proof}

Sperner's lemma can be proved by applying the KKM lemma to particular sets $X_i$; see \cite{Voorneveld2017}. So with minor changes the proof above can be rewritten to derive Sperner's lemma from the no-bullying lemma.

\bibliographystyle{abbrvnat}
\bibliography{BrouwerReferences}
\end{document}